\documentclass[11pt]{amsart}%
\usepackage{color}
\usepackage{amsmath}
\usepackage{amssymb}
\usepackage{amsfonts}
\usepackage[latin1]{inputenc}
\usepackage{graphicx}%
\providecommand{\U}[1]{\protect\rule{.1in}{.1in}}
\providecommand{\U}[1]{\protect\rule{.1in}{.1in}}

\DeclareMathSymbol{\subsetneqq}{\mathbin}{AMSb}{36}

\theoremstyle{plain}
\numberwithin{equation}{section}
\newtheorem{theorem}{Theorem}[section]

\newtheorem{lemma}{Lemma}[section]
\newtheorem{proposition}{Proposition}[section]

\newtheorem{remark}{Remark}[section]

\begin{document}
\title[Asymptotic for a semilinear hyperbolic equation with integrable source]{Asymptotic for a semilinear hyperbolic equation with asymptotically vanishing
damping term, convex potential, and integrable source}
\author{Mounir Balti}%
\address{Institut Pr\'eparatoire  aux Etude Scientifiques et Techniques, Universit\'e de Carthage, Bp 51 La Marsa, Tunisia} 
\email{mounir.balti@gmail.com}
\author{Ramzi May}%
\address{Mathematics Department, College of Science, King Faisal University, P.O. 380, Ahsaa 31982, Kingdom of Saudi Arabia}
\email{rmay@kfu.edu.sa}
.
\vskip 0.2cm
\date{Juin 28, 2022}
\maketitle

\begin{abstract} We investigate the long time behavior of solutions
to semilinear hyperbolic equations on the form:%
\begin{equation}
u^{\prime\prime}(t)+\gamma(t)u^{\prime}(t)+Au(t)+f(u(t))=g(t),~t\geq0,
\tag{E$_\alpha$}%
\end{equation}
where $A$ is a self-adjoint nonnegative operator, $f$ a function which is the gradient a regular convex function, and $\gamma$ a nonnegative function which behaviors,
for $t$ large enough, as $\frac{K}{t^{\alpha}}$ with $K>0$ and $\alpha
\in\lbrack0,1[.$ We obtain sufficient conditions on the source term $g(t),$ that ensure the weak or the strong convergence of any solution $u(t)$ of
(E$_{\alpha}$) as $t\rightarrow+\infty$ to a solution of the stationary
equation $Av+f(v)=0$ if one exists.
\end{abstract}
\section{Introduction and statement of the main results}

Let $H$ be a real Hilbert space with inner product and norm respectively
denoted by $\langle.,.\rangle$ and $\left\vert .\right\vert .$ $V$ is an other
Hilbert space continuously and densely embedded in $H.$ Let $V^{\prime}$ be the
dual space of $V.$ By identifying $H$ with its dual space, we have
$V\hookrightarrow H\hookrightarrow V^{\prime}.$ We recall the following
important relation that will be used repeatedly in the sequel:%
\begin{equation}
\langle v,w\rangle_{V^{\prime},V}=\langle v,w\rangle~\forall(v,w)\in H\times
V. \label{Kh}%
\end{equation}
In this paper, we investigate the long time behavior as $t\rightarrow+\infty$
of solutions $u(t)$ to the following second order semi-linear hyperbolic
equation:%
\begin{equation}
u^{\prime\prime}(t)+\gamma(t)u^{\prime}(t)+Au(t)+f(u(t))=g(t),~t\geq0,
\tag{E$_\alpha$}%
\end{equation}
where the damping term $\gamma$ is a function in $W_{loc}^{1,1}(\mathbb{R}%
^{+},\mathbb{R}^{+})$ which behaviors like $\frac{K}{t^{\alpha}}$ for some
$K>0,~\alpha\in\lbrack0,1[,$ and $t$ large enough. Precisely, we assume that
there exist $K>0,~t_{0}\geq0$ and $\alpha\in\lbrack0,1[$ such that:%
\begin{equation}
\gamma(t)\geq\frac{K}{(1+t)^{\alpha}}~\forall t\geq t_{0}, \label{h1}%
\end{equation}%
\begin{equation}
((1+t)^{\alpha}\gamma(t))^{\prime}\leq0~\text{a.e. }t\geq t_{0}. \label{h2}%
\end{equation}
The operator $A:V\rightarrow V^{\prime}$ is linear and continuous. We suppose
that the associated bilinear form $a:V\times V\rightarrow\mathbb{R}$ defined
by
\[
a(v,w)=\langle Av,w\rangle_{V^{\prime},V}%
\]
is symmetric, positive and satisfies the semi-coercivity property:%
\[
\exists\lambda\geq0,\mu>0:~a(v,v)+\lambda\left\vert v\right\vert ^{2}\geq
\mu\left\Vert v\right\Vert _{V}^{2}~\forall v\in V.
\]
We assume that the function $f:V\rightarrow V^{\prime}$ is a continuous
function that derives from a $C^{1}$ convex function $F:V\rightarrow
\mathbb{R}$ in the sense:
\begin{equation}
\forall u,v\in V,F^{\prime}(u)(v)=\langle f(u),v\rangle_{V^{\prime},V},
\label{ppr}%
\end{equation}
which is equivalent to 
\begin{equation}
\forall u\in V,\nabla F(u)= f(u).
\end{equation}

Let us consider the function $\Phi:V\rightarrow\mathbb{R}$ defined by:%
\[
\Phi(v)=\frac{1}{2}a(v,v)+F(v).
\]
It is clear that $\Phi$ is a $C^{1}$ convex function and for all $v\in V,$
$\nabla\Phi(v)=Av+f(v).$ We assume that the set
\[
\arg\min\Phi=\{v\in V:\Phi(v)=\min_{V}\Phi:=\Phi^{\ast}\},
\]
which coincides with the set $\{v\in V:Av+f(v)=0\},$ is nonempty.
\par\noindent Last, we suppose that function $g:\mathbb{R}^{+}\rightarrow H$ belongs to the space
$L^{1}(\mathbb{R}^{+},H).$

In this paper, we assume the existence of a global solution $u$ to Eq.
(E$_{\alpha}$) in the class%
\begin{equation}
W_{loc}^{2,1}(\mathbb{R}^{+},H)\cap W_{loc}^{1,1}(\mathbb{R}^{+},V), \label{J}%
\end{equation}
and we focus our attention on the study of the asymptotic behavior of $u(t)$ as $t$ goes to infinity.
Before setting our main theorems, let us first recall some previous results related to this subject.
In the pioneer paper \cite{Al}, Alvarez considered the case where $V=H,$ the
damping term $\gamma$ is a non negative constant and the source $g$ is equal
to $0.$ He proved that $u(t)$ converges weakly to a minimizer
of the function $\Phi.$ Moreover, he showed that the convergence is strong if
the function $\Phi$ is even or the interior of $\arg\min\Phi$ is not empty. In
\cite{HJ}, Haraux and Jendoubi extended the weak convergence result of Alvarez
to the case where the source term is in the space $L^{1}(\mathbb{R}^{+},H).$ Cabot and Frankel \cite{CF} studied Eq. (E$_{\alpha}$) where $g=0$
and $\gamma(t)$ behaviors at infinity like $\frac{K}{t^{\alpha}}$ with $K>0$
and $\alpha\in]0,1[.$ They proved that every \underline{bounded} solution
converges weakly toward a critical point of $\Phi.$ In the paper
\cite{M}, the second author of the present paper improved the result of Cabot
and Frankel by getting rid of the supplementary hypothesis on the boundedness
of the solution. In \cite{JM}, it was proved that the main convergence result
of Cabot and Frankel remains true if the source term $g$ satisfies the
condition $\int_{0}^{+\infty}(1+t)\left\vert g(t)\right\vert dt<\infty$. The
first purpose of the present paper is to improve this last result. In fact, we
prove that the convergence holds under the weaker and optimal condition
\begin{equation}
\int_{0}^{+\infty}(1+t)^{\alpha}\left\vert g(t)\right\vert dt<\infty.
\label{op}%
\end{equation}
Precisely, we establish the following result.

\begin{theorem}
\label{th1}Assume that $\int_{0}^{+\infty}(1+t)^{\alpha}\left\vert
g(t)\right\vert dt<\infty.$ Let $u$ be a solution to Eq. (E$_{\alpha}$) in the
class (\ref{J}). If $u\in L^{\infty}(\mathbb{R}^{+},H),$ then $u(t)$ converges
weakly in $V$ as $t\rightarrow+\infty$ toward some element of $\arg\min\Phi.$
Moreover, the energy function
\begin{equation}
\mathcal{E}(t):=\frac{1}{2}\left\vert u^{\prime}(t)\right\vert ^{2}%
+\Phi(u(t))-\Phi^{\ast} \label{Mn}%
\end{equation}
satisfies $\mathcal{E}(t)=\circ(t^{-2\alpha})$ as $t\rightarrow+\infty.$
\end{theorem}

In the next theorem, we prove that we can get rid of the hypothesis on the
boundedness of the solution by adding a second condition on the source term.
This theorem generalizes the main result of \cite{M}.

\begin{theorem}
\label{th2}Assume that $\int_{0}^{+\infty}(1+t)^{\alpha}\left\vert
g(t)\right\vert dt<\infty$ and $\int_{0}^{+\infty}(1+t)^{3\alpha}\left\vert
g(t)\right\vert ^{2}dt<\infty.$ Let $u$ be a solution to Eq. (E$_{\alpha}$) in
the class (\ref{J}). Then $u\in L^{\infty}(\mathbb{R}^{+},H)$ and, therefore,
we have the same conclusion as in Theorem \ref{th1}.
\end{theorem}

Our two last main results concern the strong convergence of the solution when
the potential function $\Phi$ is even or the interior of the set $\arg\min
\Phi$ is nonempty.

\begin{theorem}
\label{th3}Assume that the function $\Phi$ is even, $\int_{0}^{+\infty
}(1+t)^{\alpha}\left\vert g(t)\right\vert dt<\infty,$ and $\int_{0}^{+\infty
}(1+t)^{2\alpha+1}\left\vert g(t)\right\vert ^{2}dt<\infty.$ Let $u$ be a
solution to Eq. (E$_{\alpha}$) in the class (\ref{J}). Then there exists
$u_{\infty}\in\arg\min\Phi$ such that $u(t)\rightarrow u_{\infty}$ strongly in
$V$ as $t\rightarrow+\infty.$
\end{theorem}

\begin{theorem}
\label{th4}Assume that the interior of the set $\arg\min\Phi$ with respect of
the strong topology of $V$ is not empty. Let $u$ be a solution to Eq.
(E$_{\alpha}$) in the class (\ref{J}). If $\int_{0}^{+\infty}(1+t)^{\alpha
}\left\vert g(t)\right\vert dt<\infty$ and $u\in L^{\infty}(\mathbb{R}%
^{+},H),$ then $u(t)$ converges strongly in $V$ as $t\rightarrow+\infty$ to
some element of $\arg\min\Phi.$
\end{theorem}
\begin{remark}
In \cite{ACPR}, the authors have considered the second
order differential equation
\begin{equation}
u^{\prime\prime}(t)+\frac{K}{t}u^{\prime}(t)+\nabla\Phi(u(t))=g(t),
\tag{E$_1$}%
\end{equation}
where $K>0,$ $\Phi:H\rightarrow\mathbb{R}$ is a smooth convex function, and
$g\in L^{1}(\mathbb{R}^{+},H).$ In Eq. (E$_{1}$) the damping term
$\gamma(t)=\frac{K}{t}$ satisfies (\ref{h1}) and (\ref{h2}) with $\alpha=1$.
Hence (E$_{1}$) can be considered as a limit case of Eq. (E$_{\alpha}$). The authors supposed that the source term $g$ satisfies only the
optimal condition $\int_{0}^{+\infty}(1+t)\left\vert g(t)\right\vert
dt<\infty$ corresponding to (\ref{op}) with $\alpha=1,$ and they proved that
if $K>3$ and $\arg\min\Phi\neq\varnothing$ then every solution to (E$_{1}$)
converges weakly to a minimizer of the function $\Phi$ and strongly if in
addition $\Phi$ is even or the interior of $\arg\min\Phi$ is nonempty. It
would be interesting to know whether our main results remain true if we
suppose that the function $g$ satisfies only the optimal condition (\ref{op})
and without assuming that the solution $u$ belongs to the space $L^{\infty
}(\mathbb{R}^{+},H).$
\end{remark}

\begin{remark}
A typical example of Eq. (E$_{\alpha})$ is the following nonlinear damped wave
equation:%
\[
u_{tt}+\gamma(t)u_{t}-\Delta u+f(u)=g\text{ on }\Omega\times]0,+\infty
\lbrack,
\]
with the Dirchlet boundary condition:%
\[
u=0\text{ on }\partial\Omega\times]0,+\infty\lbrack,
\]
where $\Omega$ is a bounded open subset of $\mathbb{R}^{N},~g\in
L^{1}([0,+\infty\lbrack,L^{2}(\Omega)),$ and $f:\mathbb{R}\rightarrow
\mathbb{R}$ is a continuous and nondecreasing function which satisfies
\[
\left\vert f(s)\right\vert \leq C(1+\left\vert s\right\vert )^{m}~\forall
s\in\mathbb{R},
\]
where $C$ and $m$ are nonnegative constants with $m\leq\frac{N}{N-2}$ if
$N\geq3.$ Here $H=L^{2}(\Omega),~V=H_{0}^{1}(\Omega)$, $V^{\prime}%
=H^{-1}(\Omega),$ $a(v,w)=\int_{\Omega}\nabla v\nabla wdx$, and $F$ is the
function defined on $H_{0}^{1}(\Omega)$ by:%
\[
F(v)=\int_{\Omega}\int_{0}^{v(x)}f(s)dsdx.
\]
Using Sobolev's inequalities, one can easily verify that the function
$v\mapsto f(v)$ is continuous from $H_{0}^{1}(\Omega)$ to $L^{2}(\Omega)$ and
$F$ is a $C^{1}$ convex function which satisfies the property (\ref{ppr}), in
fact%
\[
\forall v,w\in H_{0}^{1}(\Omega),F^{\prime}(v)(w)=\int_{\Omega}f(v(x))w(x)dx.
\]
\end{remark}

\section{Preliminary results}

In this section, we prove some important preliminary results which will be
very useful in the next section to prove the main theorems.

\begin{proposition}
\label{pr1}let $u$ be a solution to Eq. (E$_{\alpha}$) in the class (\ref{J}).
Assume that there exists $\nu\in\lbrack0,1+\alpha\lbrack$ such that: $\int
_{0}^{+\infty}(1+t)^{\frac{\nu}{2}}\left\vert g(t)\right\vert dt<\infty.$
Assume moreover that $u\in L^{\infty}(\mathbb{R}^{+},H)$ or $\int_{0}%
^{+\infty}(1+t)^{\nu+\alpha}\left\vert g(t)\right\vert ^{2}dt<\infty.$Then
\[
\int_{0}^{+\infty}(1+t)^{\nu-\alpha}\left\vert u^{\prime}(t)\right\vert
^{2}dt<\infty,
\]
and the energy function $\mathcal{E}$,\ given by\ (\ref{Mn}), satisfies
$\mathcal{E}(t)=\circ(t^{-\nu})$ as $t\rightarrow$ $+\infty.$
\end{proposition}

\begin{proof}
The proof of this proposition makes use of a modified version of a method
introduced by Cabot et Frankel in \cite{CF} and developed in
\cite{M}. Let $\bar{u}\in\arg\min\Phi$ and define the function $p:\mathbb{R}%
^{+}\rightarrow\mathbb{R}^{+}$ by $p(t)=\frac{1}{2}\left\vert u(t)-\bar
{u}\right\vert ^{2}.$ Since $u$ is in the class (\ref{J}), the function $p$
belongs to the space $W_{loc}^{2,1}(\mathbb{R}^{+},\mathbb{R}^{+})$ and
satisfies almost everywhere on $\mathbb{R}^{+}$%
\begin{align}
p^{\prime\prime}(t)+\gamma(t)p^{\prime}(t)  &  =\left\vert u^{\prime
}(t)\right\vert ^{2}+\langle\nabla\Phi(u(t)),\bar{u}-u(t)\rangle+\langle
g(t),u(t)-\bar{u}\rangle\nonumber\\
&  =\left\vert u^{\prime}(t)\right\vert ^{2}+\langle\nabla\Phi(u(t)),\bar
{u}-u(t)\rangle_{V^{\prime},V}+\langle g(t),u(t)-\bar{u}\rangle\nonumber\\
&  \leq\left\vert u^{\prime}(t)\right\vert ^{2}+\Phi(\bar{u})-\Phi
(u(t))+\left\vert g(t)\right\vert \sqrt{2p(t)}\nonumber\\
&  =\frac{3}{2}\left\vert u^{\prime}(t)\right\vert ^{2}-\mathcal{E}%
(t)+\left\vert g(t)\right\vert \sqrt{2p(t)}, \label{H}%
\end{align}
where we have used (\ref{Kh}) and the convexity inequality $\Phi(\bar{u}%
)\geq\Phi(u(t))+\langle\nabla\Phi(u(t)),\bar{u}-u(t)\rangle_{V^{\prime},V}.$
On the other hand the energy function $\mathcal{E}$ belongs to $W_{loc}%
^{1,1}(\mathbb{R}^{+},\mathbb{R})$ and satisfies for almost every $t\geq0,$%
\begin{align}
\mathcal{E}^{\prime}(t)  &  =\langle u^{\prime\prime}(t),u^{\prime}%
(t)\rangle+\langle\nabla\Phi(u(t)),u^{\prime}(t)\rangle_{V^{\prime}%
,V}\nonumber\\
&  =\langle u^{\prime\prime}(t),u^{\prime}(t)\rangle+\langle\nabla
\Phi(u(t)),u^{\prime}(t)\rangle\nonumber\\
&  =-\gamma(t)\left\vert u^{\prime}\right\vert ^{2}+\langle g(t),u^{\prime
}(t)\rangle. \label{He}%
\end{align}
For every $r\in\mathbb{R},$ we define the function $\lambda_{r}$ on
$\mathbb{R}^{+}$ by $\lambda_{r}(t)=(1+t)^{r}.$ In view of (\ref{He}),%
\begin{equation}
(\lambda_{\nu}\mathcal{E})^{\prime}=\lambda_{\nu}^{\prime}\mathcal{E}%
-\lambda_{\nu}\gamma\left\vert u^{\prime}\right\vert ^{2}+\lambda_{\nu}\langle
g,u^{\prime}\rangle. \label{hedi1}%
\end{equation}
Hence,%
\begin{equation}
\lambda_{\nu}\gamma\left\vert u^{\prime}\right\vert ^{2}\leq\lambda_{\nu
}^{\prime}\mathcal{E}-(\lambda_{\nu}\mathcal{E})^{\prime}+\lambda_{\frac{\nu
}{2}}\left\vert g\right\vert \sqrt{2\lambda_{\nu}\mathcal{E}}. \label{Hedi2}%
\end{equation}
Since $\gamma$ satisfies (\ref{h1}) with $\alpha<1,$ $\lambda_{\nu}^{\prime
}(t)\left\vert u^{\prime}(t)\right\vert ^{2}=\circ(\lambda_{\nu}%
(t)\gamma(t)\left\vert u^{\prime}(t)\right\vert ^{2})$ as $t\rightarrow
+\infty.$ Then there exists $t_{1}\geq t_{0}$ such that%
\begin{equation}
\frac{3}{2}\lambda_{\nu}^{\prime}(t)\left\vert u^{\prime}(t)\right\vert
^{2}\leq\frac{1}{2}\lambda_{\nu}(t)\gamma(t)\left\vert u^{\prime
}(t)\right\vert ^{2}\text{ a.e. }t\geq t_{1}. \label{HH}%
\end{equation}
Thus, by multiplying the inequality (\ref{H}) by $\lambda_{\nu}^{\prime}(t)$
and using (\ref{Hedi2})-(\ref{HH}), we obtain%
\[
\frac{1}{2}\lambda_{\nu}^{\prime}\mathcal{E}+\frac{1}{2}(\lambda_{\nu
}\mathcal{E)}^{\prime}\leq-\lambda_{\nu}^{\prime}p^{\prime\prime}-\lambda
_{\nu}^{\prime}\gamma p^{\prime}+\lambda_{\nu}^{\prime}\left\vert g\right\vert
\sqrt{2p}+\frac{1}{2}\lambda_{\frac{\nu}{2}}\left\vert g\right\vert
\sqrt{2\lambda_{\nu}\mathcal{E}}\text{,}%
\]
almost everywhere on $[t_{1},\infty\lbrack.$
\par\noindent Integrating this last inequality
between $t_{1}$ and $t\geq t_{1},$ we get after integrations by parts%
\begin{equation}
\frac{1}{2}\int_{t_{1}}^{t}\lambda_{\nu}^{\prime}\mathcal{E}ds+\frac{1}%
{2}(\lambda_{\nu}\mathcal{E})(t)\leq C_{0}+A(t)+B(t)+C(t), \label{Rz}%
\end{equation}
where%
\begin{align*}
C_{0}  &  =\frac{1}{2}(\lambda_{\nu}\mathcal{E)}(t_{1})+(\lambda_{\nu}%
^{\prime}p^{\prime})(t_{1})-(\lambda_{\nu}^{\prime\prime}p)(t_{1}%
)+(\lambda_{\nu}^{\prime}\gamma p)(t_{1}),\\
A(t)  &  =-(\lambda_{\nu}^{\prime}p^{\prime})(t)+(\lambda_{\nu}^{\prime\prime
}p)(t)-(\lambda_{\nu}^{\prime}\gamma p)(t),\\
B(t)  &  =\int_{t_{1}}^{t}(-\lambda_{\nu}^{(3)}+(\lambda_{\nu}^{\prime}%
\gamma)^{\prime})p+\lambda_{\nu}^{\prime}\left\vert g\right\vert \sqrt
{2p}ds,\\
C(t)  &  =\int_{t_{1}}^{t}\lambda_{\frac{\nu}{2}}\left\vert g\right\vert
\sqrt{\lambda_{\nu}\mathcal{E}}ds.
\end{align*}
Let us estimate separately $A(t),B(t),$ and $C(t).$ Firstly, by using the fact
that $\sqrt{\lambda_{\nu}\mathcal{E}}\leq1+\lambda_{\nu}\mathcal{E},$ we get%
\begin{equation}
C(t)\leq\int_{0}^{+\infty}(1+s)^{\frac{\nu}{2}}\left\vert g(s)\right\vert
ds+\int_{t_{1}}^{t}\lambda_{\frac{\nu}{2}}\left\vert g\right\vert \lambda
_{\nu}\mathcal{E}ds. \label{Meriem1}%
\end{equation}
On the other hand, in view of (\ref{h1})%
\begin{align*}
A(t)  &  \leq\lambda_{\nu}^{\prime}(t)\left\vert \langle u^{\prime
}(t),u(t)-\bar{u}\rangle\right\vert -\nu\lbrack K-(\nu-1)(1+t)^{\alpha
-1}](1+t)^{\nu-\alpha-1}p(t)\\
&  \leq2\lambda_{\nu}^{\prime}(t)\sqrt{\mathcal{E(}t\mathcal{)}}\sqrt
{p(t)}-\nu\lbrack K-(\nu-1)(1+t)^{\alpha-1}](1+t)^{\nu-\alpha-1}p(t).
\end{align*}
Therefore, since $\alpha<1,$ there exists $t_{2}\geq t_{1}$ such that for
every $t\geq t_{2},$%
\[
A(t)\leq2\lambda_{\nu}^{\prime}(t)\sqrt{\mathcal{E(}t\mathcal{)}}\sqrt
{p(t)}-\frac{\nu K}{2}(1+t)^{\nu-\alpha-1}p(t).
\]
Using now the elementary inequality%
\begin{equation}
\forall a>0~\forall b,x\in\mathbb{R},~-ax^{2}+bx\leq\frac{b^{2}}{4a}
\label{Meriem2}%
\end{equation}
with $x=\sqrt{p(t)},$ we get%
\[
A(t)\leq\frac{2\nu}{K}(1+t)^{\nu+\alpha-1}\mathcal{E(}t)~\forall t\geq t_{2}.
\]
Using once again the fact that $\alpha<1,$ we infer the existence of
$t_{3}\geq t_{2}$ such that
\begin{equation}
A(t)\leq\frac{1}{4}\lambda_{\nu}(t)\mathcal{E}(t)~\forall t\geq t_{3}.
\label{Hedi 3}%
\end{equation}
Let us now prove that the function $B$ is bounded. To this end we first notice
that, thanks to (\ref{h1}) and (\ref{h2}), we have for almost every $t\geq t_{1}$%
\begin{align*}
-\lambda_{\nu}^{(3)}(t)+(\lambda_{\nu}^{\prime}\gamma)^{\prime}(t)  &
\leq-\lambda_{\nu}^{(3)}(t)+\lambda_{\nu}^{\prime\prime}\gamma-\alpha
\lambda_{\nu}^{\prime}(t)\frac{\gamma(t)}{(1+t)}\\
&  \leq-\lambda_{\nu}^{(3)}(t)-\nu K(1+\alpha-\nu)(1+t)^{\nu-2-a}.
\end{align*}
Since $\nu<1+\alpha$ and $\alpha<1,$ there exists $t_{4}\geq t_{3}$ such that
for almost every $t\geq t_{4},$%
\begin{equation}
-\lambda_{\nu}^{(3)}(t)+(\lambda_{\nu}^{\prime}\gamma)^{\prime}(t)\leq
-\mu(1+t)^{\nu-2-a}, \label{Hedi 4}%
\end{equation}
where $\mu=\frac{\nu K(1+\alpha-\nu)}{2}>0.$ Therefore, if $u\in L^{\infty
}(\mathbb{R}^{+},H)$ then for every $t\geq t_{4}$ we have
\begin{align*}
B(t)  &  \leq B(t_{4}) +\sqrt{\sup_{t\geq0}2p(t)}\int_{0}^{+\infty}%
\lambda_{\nu}^{\prime}\left\vert g\right\vert dt\\
&  \leq B(t_{4}) +\nu\sqrt{\sup_{t\geq0}2p(t)}\int_{0}^{+\infty}%
(1+t)^{\frac{\nu}{2}}\left\vert g\right\vert dt
\end{align*}
Let us now examine the boundedness of the function $B$ under the other
hypothesis $\int_{0}^{+\infty}(1+t)^{\nu+\alpha}\left\vert g(t)\right\vert
^{2}dt<\infty.$ By using (\ref{Hedi 4}) and the inequality (\ref{Meriem2})
with $x=\sqrt{p(t)}$ we easily get that for every $t\geq t_{4}$%
\begin{align*}
B(t)  &  \leq B\left(  t_{4}\right)  +\frac{2\nu^{2}}{\mu}\int_{t_{4}}%
^{t}(1+s)^{\nu+\alpha}\left\vert g(s)\right\vert ^{2}dt\\
&  \leq B\left(  t_{4}\right)  +\frac{2\nu^{2}}{\mu}\int_{0}^{+\infty
}(1+s)^{\nu+\alpha}\left\vert g(s)\right\vert ^{2}dt.
\end{align*}
Coming back to (\ref{Rz}) and using the estimates (\ref{Meriem1}%
)-(\ref{Hedi 3}) and the boundedness of the function $B,$ we infer the
existence of a constant $C_{1}\geq0$ such that for every $t\geq t_{4},$%
\[
\frac{1}{2}\int_{t_{1}}^{t}\lambda_{\nu}^{\prime}\mathcal{E}ds+\frac{1}%
{4}(\lambda_{\nu}\mathcal{E})(t)\leq C_{1}+\int_{t_{1}}^{t}\lambda_{\frac{\nu
}{2}}\left\vert g\right\vert \lambda_{\nu}\mathcal{E}dt.
\]
Therefore, by applying Gronwall's inequality we first get that $\sup_{t\geq
t_{1}}\lambda_{\nu}(t)\mathcal{E(}t\mathcal{)}<+\infty$ and then we deduce
that $\int_{t_{1}}^{+\infty}\lambda_{\nu}^{\prime}(t)\mathcal{E}(t)dt<+\infty
$. Recalling that the energy function $\mathcal{E}$ is continuous and hence
locally bounded on\thinspace$\mathbb{R}^{+},$ we infer that
\begin{equation}
\int_{0}^{+\infty}\lambda_{\nu}^{\prime}(t)\mathcal{E}(t)dt<+\infty\label{R}%
\end{equation}
and%
\begin{equation}
\sup_{t\geq0}\lambda_{\nu}(t)\mathcal{E(}t\mathcal{)}<+\infty. \label{r}%
\end{equation}
Hence by using the equality (\ref{hedi1}) we obtain
\begin{align*}
\int_{0}^{+\infty}[(\lambda_{\nu}\mathcal{E})^{\prime}]_{+}dt  &  \leq\int
_{0}^{+\infty}\lambda_{\nu}^{\prime}\mathcal{E}dt+\sqrt{\sup_{t\geq0}%
2\lambda_{\nu}(t)\mathcal{E}(t)}\int_{0}^{+\infty}\mathcal{\lambda}_{\frac
{\nu}{2}}\left\vert g\right\vert dt\\
&  <+\infty,
\end{align*}
where $[(\lambda_{\nu}\mathcal{E})^{\prime}]_{+}$ is the positive part of
$(\lambda_{\nu}\mathcal{E})^{\prime}.$ The last inequality implies that
$\lambda_{\nu}(t)\mathcal{E(}t\mathcal{)}$ converges as $t$ goes to $+\infty$
to some real number $m.$ If $m\neq0$ then $\lambda_{\nu}^{\prime
}(t)\mathcal{E}(t)=\frac{\lambda_{\nu}(t)\mathcal{E(}t\mathcal{)}}{\nu
(1+t)}\sim\frac{m}{\nu(1+t)}$ as $t\rightarrow+\infty$ which contradicts the
result (\ref{R}). Thus $m=0$ and therefore $\mathcal{E}(t)=\circ(t^{-\nu})$ as
$t\rightarrow+\infty.$ Finally, by using the inequality (\ref{Hedi2}), we
obtain
\[
\int_{0}^{+\infty}\lambda_{\nu}\gamma\left\vert u^{\prime}\right\vert
^{2}dt\leq\int_{0}^{+\infty}\lambda_{\nu}^{\prime}\mathcal{E}dt+\mathcal{E}%
(0)+\sqrt{\sup_{t\geq0}2\lambda_{\nu}(t)\mathcal{E}(t)}\int_{0}^{+\infty
}\mathcal{\lambda}_{\frac{\nu}{2}}\left\vert g\right\vert dt.
\]
In view of (\ref{R}) and (\ref{r}), the right hand side of the previous
inequality is finite, then thanks to the hypothesis (\ref{h1}) we conclude
that
\[
\int_{0}^{+\infty}(1+t)^{\nu-\alpha}\left\vert u^{\prime}(t)\right\vert
^{2}dt<+\infty
\]
as desired.
\end{proof}

\begin{proposition}
\label{pr2}Let $u$ be a solution to Eq. (E$_{\alpha}$). Assume that the
integrals $\int_{0}^{+\infty}(1+t)^{\alpha}\left\vert g(t)\right\vert dt$ and
$\int_{0}^{+\infty}(1+t)^{\alpha}\left\vert u^{\prime}(t)\right\vert ^{2}dt$
are finite and $\Phi(u(t))\rightarrow\Phi^{\ast}$ as $t\rightarrow+\infty.$
Then $u(t)$ converges weakly in $V$ as $t\rightarrow+\infty$ toward some
element $u_{\infty}$ of $\arg\min\Phi.$
\end{proposition}

The proof of this proposition repose on the classical Opial's lemma \cite{Op}
(see \cite{AGR} for a simple proof) and an elementary lemma which will be also
used to prove Theorem \ref{th3} and Theorem \ref{th4}. Let us first recall
Opial's lemma.

\begin{lemma}
[Opial's lemma]Let $x:[t_{0},+\infty\lbrack\rightarrow\mathcal{H}.$ Assume
that there exists a nonempty subset $S$ of $\mathcal{H}$ such that:

\begin{enumerate}
\item[(i)] If $t_{n}\rightarrow+\infty$ and $x(t_{n})\rightharpoonup x$ weakly
in $\mathcal{H}$ , then $x\in S.$

\item[(ii)] For every $z\in S,$ $\lim_{t\rightarrow+\infty}\left\Vert
x(t)-z\right\Vert $ exists.
\end{enumerate}
\noindent Then there exists $z_{\infty}\in S$ such that $x(t)\rightharpoonup
z_{\infty}$ weakly in $\mathcal{H}$ as $t\rightarrow+\infty.$
\end{lemma}

\begin{lemma}
\label{le1}There exists $\tau_{0}\geq0$ such that for every $\tau\geq\tau_{0}$%
\[
\int_{\tau}^{+\infty}e^{-\Gamma(t,\tau)}dt\leq\frac{2}{K}(1+\tau)^{\alpha}%
\]
where $\Gamma(t,\tau)=\int_{\tau}^{t}\gamma(s)ds.$
\end{lemma}

\begin{proof}
Let $\tau\geq t_{0}.$ In view of (\ref{h1}),
\begin{align*}
\int_{\tau}^{+\infty}e^{-\Gamma(t,\tau)}dt  &  \leq\frac{1}{K}\int_{\tau
}^{+\infty}(1+t)^{\alpha}\gamma(t)e^{-\Gamma(t,\tau)}dt\\
&  =-\frac{1}{K}\int_{\tau}^{+\infty}(1+t)^{\alpha}\left(  e^{-\Gamma(t,\tau
)}\right)  ^{\prime}dt\\
&  =\frac{1}{K}(1+\tau)^{\alpha}+\frac{\alpha}{K}\int_{\tau}^{+\infty
}(1+t)^{\alpha-1}e^{-\Gamma(t,\tau)}dt\\
&  \leq\frac{1}{K}(1+\tau)^{\alpha}+\frac{\alpha}{K(1+\tau)^{1-\alpha}}%
\int_{\tau}^{+\infty}e^{-\Gamma(t,\tau)}dt.
\end{align*}
Hence to conclude we just have to choose $\tau_{0}$ large enough such that
$\frac{\alpha}{K(1+\tau_{0})^{1-\alpha}}\leq\frac{1}{2}.$
\end{proof}

\begin{proof}
[Proof of Proposition \ref{pr2} ]Let us first prove that $u\in L^{\infty
}(\mathbb{R}^{+},V).$ Let $\bar{u}\in\arg\min\Phi$ and define, as in the proof
of Proposition \ref{pr1}, the function $p:\mathbb{R}^{+}\rightarrow
\mathbb{R}^{+}$ by $p(t)=\frac{1}{2}\left\vert u(t)-\bar{u}\right\vert ^{2}.$
This function belongs to the space $W_{loc}^{2,1}(\mathbb{R}^{+}%
,\mathbb{R}^{+})$ and satisfies almost everywhere on $\mathbb{R}^{+},$%
\begin{align*}
p^{\prime\prime}+\gamma p^{\prime}  &  =\left\vert u^{\prime}\right\vert
^{2}-\langle\nabla\Phi(u),u-\bar{u}\rangle+\langle g,u-\bar{u}\rangle\\
&  =\left\vert u^{\prime}\right\vert ^{2}-\langle\nabla\Phi(u)-\nabla\Phi
(\bar{u}),u-\bar{u}\rangle_{V^{\prime},V}+\langle g,u-\bar{u}\rangle\\
&  \leq\left\vert u^{\prime}\right\vert ^{2}+\left\vert g\right\vert \sqrt
{2p},
\end{align*}
where we have used the monotonicity of the operator $\nabla\Phi.$ 
\noindent Therefore, for almost every $t\geq\tau_{0},$%
\[
p^{\prime}(t)\leq e^{-\Gamma(t,\tau_{0})}p^{\prime}(\tau_{0})+\int_{\tau_{0}%
}^{t}e^{-\Gamma(t,s)}\rho(s)ds
\]
where $\rho:=\left\vert u^{\prime}\right\vert ^{2}+\left\vert g\right\vert
\sqrt{2p}.$

\noindent Thus, by using the previous lemma and Fubini's theorem, we get for
every $t\geq\tau_{0}$%
\begin{align}
\int_{\tau_{0}}^{t}[p^{\prime}(\tau)]_{+}d\tau &  \leq\frac{2\left(
1+\tau_{0}\right)  ^{\alpha}}{K}\left\vert p^{\prime}(\tau_{0})\right\vert
+\frac{2}{K}\int_{\tau_{0}}^{t}(1+s)^{\alpha}\rho(s)ds\nonumber\\
&  \leq c_{0}+\frac{2}{K}\int_{\tau_{0}}^{t}(1+s)^{\alpha}\left\vert
g(s)\right\vert \sqrt{2p(s)}ds \label{N1}%
\end{align}
where $c_{0}=\frac{2\left(  1+\tau_{0}\right)  ^{\alpha}}{K}\left\vert
p^{\prime}(\tau_{0})\right\vert +\frac{2}{K}\int_{0}^{+\infty}(1+s)^{\alpha
}\left\vert u^{\prime}(s)\right\vert ^{2}ds$ and $[p^{\prime}(\tau)]_{+}$ is
the positive part of $p^{\prime}(\tau).$

\noindent Using now the inequalities $\sqrt{2p}\leq1+2p$ and $p(t)\leq
p(\tau_{0})+\int_{\tau_{0}}^{t}[p^{\prime}(\tau)]_{+}d\tau,$ we obtain%
\[
p(t)\leq c_{1}+\frac{4}{K}\int_{\tau_{0}}^{t}(1+s)^{\alpha}\left\vert
g(s)\right\vert p(s)ds,~\forall t\geq\tau_{0},
\]
with $c_{1}=c_{0}+p(\tau_{0})+\frac{2}{K}\int_{0}^{+\infty}(1+s)^{\alpha
}\left\vert g(s)\right\vert ds.$ Hence, by applying Gronwall's inequality, we
deduce that the function $p$ is bounded which is equivalent to $u\in
L^{\infty}(\mathbb{R}^{+},H).$ Using now \cite[Remark 3.4]{CF}, we obtain that
$u\in L^{\infty}(\mathbb{R}^{+},V).$ Coming back to the estimate (\ref{N1}),
we infer that%
\begin{align*}
\int_{\tau_{0}}^{+\infty}[p^{\prime}(\tau)]_{+}d\tau &  \leq c_{0}+\frac{2}%
{K}\int_{0}^{+\infty}(1+s)^{\alpha}\left\vert g(s)\right\vert ds\sqrt
{\sup_{t\geq0}2p(t)}\\
&  <+\infty
\end{align*}
which implies that $\lim_{t\rightarrow+\infty}p(t)$ and therefore
$\lim_{t\rightarrow+\infty}\left\vert u(t)-\bar{u}\right\vert $ exist. Now,
let $\bar{x}\in H$ such that there exists a sequence $(t_{n})_{n}$ of positive
real numbers tending to $+\infty$ such that $u(t_{n})$ converges weakly in $H$
to $\bar{x}.$ Since $u\in L^{\infty}(\mathbb{R}^{+},V),$ $u(t_{n})$ converges
weakly also in the space $V$ to the same element $\bar{x}.$ Using now the weak
lower semi-continuity of the continuous and convex function $\Phi,$ we deduce
that $\Phi^{\ast}=\lim\inf\Phi(u(t_{n}))\leq\Phi(\bar{x}).$ Thus $\bar{x}%
\in\arg\min\Phi.$ Therefore, applying Opial's lemma with $S=\arg\min\Phi$
ensures that $u(t)$ converges weakly in $H$ as $t\rightarrow+\infty$ to some
element of $\arg\min\Phi.$ Recalling that $u\in L^{\infty}(\mathbb{R}^{+},V),$
we conclude that this weak convergence holds also in the space $V.$
\end{proof}
We close this section by proving the following simple lemma that will be used in the proof of Theorem \ref{th4}.
\begin{lemma}\label{Chill}
For every $v\in V,$
\begin{equation}
\left\vert v\right\vert \leq \left\Vert v\right\Vert
_{V^{\prime}}^{\frac{1}{2}}\left\Vert v\right\Vert _{V}^{\frac{1}{2}}.
\end{equation}
\end{lemma}
\begin{proof}
Let $v\in V$. From (\ref{Kh}),
$$\left\vert v\right\vert^2=\langle v,v\rangle=\langle v,v\rangle_{V^{\prime},V}\leq \left\Vert v\right\Vert
_{V^{\prime}}\left\Vert v\right\Vert _{V}.$$
The proof is then achieved.
\end{proof}

\section{Proof of the main results}

This section is devoted to the proof of our main theorems. Let us first notice
that Theorem \ref{th1} and Theorem \ref{th2} follow immediately from
Proposition \ref{pr1} (with $\nu=2\alpha$) and Proposition \ref{pr2}. Hence it remains to prove Theorem \ref{th3} and Theorem \ref{th4}

\begin{proof}
[Proof of Theorem \ref{th3}]The proof is based on the adaptation of a method
introduced by Bruck \cite{Br}\ for the steepest descent method and used by
Alvarez \cite{Al}\ for the heavy ball with friction system.

\noindent Since $2\alpha+1\geq3\alpha,$ then in view of Theorem \ref{th2},
Proposition \ref{pr1}, and Proposition \ref{pr2}, $u(t)$ converges weakly in
$V$ to some $u_{\infty}\in\arg\min\Phi,$ and $\int_{0}^{+\infty}(1+t)^{\alpha
}\left\vert u^{\prime}(t)\right\vert ^{2}dt<\infty.$ Let $\tau\geq\tau_{0}$
where $\tau_{0}$ is the real defined in Lemma \ref{le1}. We define the
function $q$ on the interval $[\tau_{0},\tau]$ by:
\[
q(t)=\left\vert u(t)\right\vert ^{2}-\left\vert u(\tau)\right\vert ^{2}%
-\frac{1}{2}\left\vert u(t)-u(\tau)\right\vert ^{2}.
\]
The function $q$ belongs to the space $W^{2,1}([\tau_{0},\tau],\mathbb{R})$
and satisfies almost everywhere%
\begin{equation}
q^{\prime}(t)=\langle u^{\prime}(t),u(t)+u(\tau)\rangle\label{Mo1}%
\end{equation}%
\begin{equation}
q^{\prime\prime}(t)=\left\vert u^{\prime}(t)\right\vert ^{2}+\langle
u^{\prime\prime}(t),u(t)+u(\tau)\rangle. \label{Mo2}%
\end{equation}
Combining this two equalities, we obtain%
\begin{align}
q^{\prime\prime}(t)+\gamma(t)q^{\prime}(t)  &  =\left\vert u^{\prime
}(t)\right\vert ^{2}+\langle\nabla\Phi(u),-u(\tau)-u(t)\rangle_{V^{\prime}%
,V}+\langle g(t),u(t)+u(\tau)\rangle\nonumber\\
&  \leq\left\vert u^{\prime}(t)\right\vert ^{2}+\Phi(-u(\tau))-\Phi
(u(t))+2M\left\vert g(t)\right\vert \nonumber\\
&  =\left\vert u^{\prime}(t)\right\vert ^{2}+\Phi(u(\tau))-\Phi
(u(t))+2M\left\vert g(t)\right\vert \nonumber\\
&  =\frac{3}{2}\left\vert u^{\prime}(t)\right\vert ^{2}+\mathcal{\tilde{E}%
}(\tau)-\mathcal{\tilde{E}}(t)+2M\left\vert g(t)\right\vert +\int_{t}^{\tau
}\frac{\left\vert g(s)\right\vert ^{2}}{4\gamma(s)}ds \label{M2}%
\end{align}
where $M=\sup_{t\geq0}\left\vert u(t)\right\vert $ and $\mathcal{\tilde{E}}$
is the modified energy function defined by:%
\[
\mathcal{\tilde{E}}(t)=\mathcal{E}(t)+\int_{t}^{+\infty}\frac{\left\vert
g(s)\right\vert ^{2}}{4\gamma(s)}ds,
\]
where $\mathcal{E}$ is the energy function given by (\ref{Mn}). Using
(\ref{He}), we get%
\begin{align*}
\mathcal{\tilde{E}}^{\prime}(t)  &  =-\gamma(t)\left\vert u^{\prime
}(t)\right\vert ^{2}+\langle g(t),u^{\prime}(t)\rangle-\frac{\left\vert
g(t)\right\vert ^{2}}{4\gamma(t)}\\
&  \leq-\left(  \sqrt{\gamma(t)}\left\vert u^{\prime}(t)\right\vert
-\frac{\left\vert g(t)\right\vert }{2\sqrt{\gamma(t)}}\right)  ^{2}.
\end{align*}
Therefore the function $\mathcal{\tilde{E}}$ is non increasing. Hence
(\ref{M2}) and (\ref{h1}) yield
\[
q^{\prime\prime}(t)+\gamma(t)q^{\prime}(t)\leq\omega(t),
\]
where
\[
\omega(t)=\frac{3}{2}\left\vert u^{\prime}(t)\right\vert ^{2}+2M\left\vert
g(t)\right\vert +\frac{1}{4K}\int_{t}^{+\infty}(1+s)^{\alpha}\left\vert
g(s)\right\vert ^{2}ds.
\]
Therefore, for almost every $t\in\lbrack\tau_{0},\tau],$
\begin{equation}
q^{\prime}(t)\leq e^{-\Gamma(t,\tau_{0})}\left\vert q^{\prime}(\tau
_{0})\right\vert +\int_{\tau_{0}}^{t}e^{-\Gamma(t,s)}\omega(s)ds\equiv
\kappa(t). \label{S}%
\end{equation}
A simple calculation, using Fubini's theorem and Lemma \ref{le1}, gives%
\[
\int_{\tau_{0}}^{+\infty}\kappa(t)dt\leq c_{0}+\frac{2}{K}\int_{\tau_{0}%
}^{+\infty}(1+s)^{\alpha}\omega(s)ds
\]
where $c_{0}=\frac{2}{K}(1+\tau_{0})^{\alpha}\left\vert q^{\prime}(\tau
_{0})\right\vert .$ 
\par\noindent Using once again Fubini's theorem, we get%
\[
\int_{\tau_{0}}^{+\infty}(1+t)^{\alpha}\int_{t}^{+\infty}(1+s)^{\alpha
}\left\vert g(s)\right\vert ^{2}dsdt\leq\frac{1}{\alpha+1}\int_{\tau_{0}%
}^{+\infty}(1+s)^{2\alpha+1}\left\vert g(s)\right\vert ^{2}ds.
\]
Then we deduce that the integral $\int_{\tau_{0}}^{+\infty}(1+s)^{\alpha
}\omega(s)ds$ is finite which implies
\begin{equation}
\int_{\tau_{0}}^{+\infty}\kappa(t)dt<+\infty. \label{es}%
\end{equation}
Integrating now (\ref{S}) between $t$ and $\tau,$ with $\tau_{0}\leq t\leq
\tau,$ we get%
\begin{equation}
\frac{1}{2}\left\vert u(t)-u(\tau)\right\vert ^{2}\leq\left\vert
u(t)\right\vert ^{2}-\left\vert u(\tau)\right\vert ^{2}+\int_{t}^{\tau}%
\kappa(s)ds. \label{est}%
\end{equation}
In the proof of Proposition \ref{pr2}, we showed that $\lim_{t\rightarrow
+\infty}\left\vert u(t)-\bar{u}\right\vert ^{2}$ exists\ for all $\bar{u}$ in
$\arg\min\Phi.$ But $0\in\arg\min\Phi$ since $\Phi$ is convex and even, then
$\lim_{t\rightarrow+\infty}\left\vert u(t)\right\vert ^{2}$ exists. Therefore,
(\ref{est}) and (\ref{es}) imply
\[
\left\vert u(\tau)-u(t)\right\vert \rightarrow0\text{ as }t,\tau
\rightarrow+\infty.
\]
Thus, in view of Cauchy criteria, $u(t)$ converges strongly in $H$ as
$t\rightarrow+\infty$ . Therefore, by using \cite[Corollary 3.6]{CF}, we
deduce that $u(t)$ converges strongly in $V$ as $t\rightarrow+\infty.$
Finally, since $u(t)\rightharpoonup u_{\infty}$ weakly in $V,$ we conclude
that $u(t)\rightarrow u_{\infty}$ strongly in $V.$
\end{proof}

\begin{proof}
[Proof of Theorem \ref{th4}] By assumption, there exists $x^{\ast}\in\arg\min\Phi$ and $r>0$ such
that for all $v$ in the unit Ball $B_{V}(0,1)$ of $V$ we have $\nabla
\Phi(x^{\ast}+rv)=0.$ Therefore the monotonicity of $\nabla\Phi$ implies that
for every $x\in V,~\langle\nabla\Phi(x),x-x^{\ast}-rv\rangle_{V^{\prime}%
,V}\geq0$ which yields that $\langle\nabla\Phi(x),v\rangle_{V^{\prime},V}%
\leq\frac{1}{r}\langle\nabla\Phi(x),x-x^{\ast}\rangle_{V^{\prime},V}.$ Hence
by taking the supremum on $v\in B_{V}(0,1),$ we get%
\begin{equation}
\left\Vert \nabla\Phi(x)\right\Vert _{V^{\prime}}\leq\frac{1}{r}\langle
\nabla\Phi(x),x-x^{\ast}\rangle_{V^{\prime},V}. \label{Q}%
\end{equation}
Let us now define the function $p(t)=\frac{1}{2}\left\vert u(t)-x^{\ast
}\right\vert ^{2}.$ We already know that $p$ satisfies the differential
inequality%
\[
p^{\prime\prime}(t)+\gamma(t)p^{\prime}(t)\leq\left\vert u^{\prime
}(t)\right\vert ^{2}-\langle\nabla\Phi(u(t)),u(t)-x^{\ast}\rangle_{V^{\prime
},V}+\langle g(t),u(t)-x^{\ast}\rangle.
\]
Hence by using (\ref{Q}), we obtain%
\begin{equation}
r\left\Vert \nabla\Phi(u(t))\right\Vert _{V^{\prime}}\leq-p^{\prime\prime
}(t)-\gamma(t)p^{\prime}(t)+\sigma(t), \label{Z}%
\end{equation}
where $\sigma(t)=\left\vert u^{\prime}(t)\right\vert ^{2}+\left\vert
g(t)\right\vert \sup_{t\geq0}\left\vert u(t)-x^{\ast}\right\vert .$

\noindent Recalling that in view Proposition \ref{pr1}, $\int_{0}^{+\infty
}\lambda_{\alpha}(t)\sigma(t)dt<\infty$ where $\lambda_{\alpha}%
(t)=(1+t)^{\alpha}.$ Hence, by multiplying (\ref{Z}) by $\lambda_{\alpha}(t)$
and integrating between $t_{0}$ and $\tau\geq t_{0},$ we get after
integrations by parts and simplification%
\begin{align*}
r\int_{t_{0}}^{\tau}\lambda_{\alpha}(t)\left\Vert \nabla\Phi(u(t))\right\Vert
_{V^{\prime}}dt  &  \leq C-\lambda_{\alpha}(\tau)p^{\prime}(\tau
)+\lambda_{\alpha}^{\prime}(\tau)p(\tau)\\
&  -\underset{\geq0}{\underbrace{(\lambda_{\alpha}\gamma)}}(\tau)p(\tau
)+\int_{t_{0}}^{\tau}[\underset{\leq0}{\underbrace{(\lambda_{\alpha}%
\gamma)^{\prime}}}-\lambda_{\alpha}^{\prime\prime}](t)p(t)dt
\end{align*}
where $C$ is a constant independent of $\tau.$ Since $\alpha<1$ and $u\in
L^{\infty}(\mathbb{R}^{+},H),$ the integral $\int_{t_{0}}^{+\infty}\left\vert
\lambda_{\alpha}^{\prime\prime}(t)\right\vert p(t)dt$ and the supremum
$\sup_{\tau\geq t_{0}}\lambda_{\alpha}^{\prime}(\tau)p(\tau)$ are finite.
Moreover, from Proposition \ref{pr1}, $\left\vert u^{\prime}(\tau)\right\vert
=\circ(\tau^{-\alpha})$ as $\tau\rightarrow+\infty,$ then
\[
\sup_{\tau\geq t_{0}}\lambda_{\alpha}(\tau)\left\vert p^{\prime}%
(\tau)\right\vert \leq\sup_{\tau\geq t_{0}}\lambda_{\alpha}(\tau)\left\vert
u^{\prime}(\tau)\right\vert \left\vert u(\tau)-x^{\ast}\right\vert <\infty.
\]
Therefore, we conclude that
\begin{equation}
\int_{t_{0}}^{+\infty}\lambda_{\alpha}(t)\left\Vert \nabla\Phi
(u(t))\right\Vert _{V^{\prime}}dt<+\infty. \label{T}%
\end{equation}
From Eq. (E$_{\alpha}$), we have
\[
u^{\prime\prime}(t)+\gamma(t)u^{\prime}(t)=g(t)-\nabla\Phi(u(t))
\]
Hence, by integrating this equation we get%
\begin{equation}
u^{\prime}(t)=e^{-\Gamma(t,\tau_{0})}u^{\prime}(\tau_{0})+\int_{\tau_{0}}%
^{t}e^{-\Gamma(t,s)}[g(s)-\nabla\Phi(u(s))]ds, \label{W}%
\end{equation}
for almost every $t\geq\tau_{0}$ where $\tau_{0}$ is the real defined by Lemma
\ref{le1}. Up to replace $\tau_{0}$ by $\tau_{0}^{\prime}>\tau_{0},$ we can
assume that $u^{\prime}(\tau_{0})\in H.$ Thus by applying Lemma \ref{le1} and
Fubini's theorem to the equality (\ref{W}), we obtain%
\begin{align*}
\int_{\tau_{0}}^{+\infty}\left\Vert u^{\prime}(t)\right\Vert _{V^{\prime}}dt
&  \leq\frac{2}{K}(1+\tau_{0})^{\alpha}\left\Vert u^{\prime}(\tau
_{0})\right\Vert _{V^{\prime}}+\frac{2}{K}\int_{\tau_{0}}^{+\infty
}(1+s)^{\alpha}\left\Vert g(s)\right\Vert _{V^{\prime}}\\
&  +\frac{2}{K}\int_{\tau_{0}}^{+\infty}(1+s)^{\alpha}\left\Vert \nabla
\Phi(u(s))\right\Vert _{V^{\prime}}ds.
\end{align*}
Hence $\int_{\tau_{0}}^{+\infty}\left\Vert u^{\prime}(t)\right\Vert
_{V^{\prime}}dt<+\infty$ thanks to the continuous injection $H\hookrightarrow
V^{\prime},$ the hypothesis on $g,$ and the estimate (\ref{T}). Thus we deduce
that $u(t)$ converges strongly in $V^{\prime}$ as $t\rightarrow+\infty$ to
some $u_{\infty}.$ Recalling that, in view of Theorem \ref{th1}, $u\in
L^{\infty}(\mathbb{R}^{+},V)$ and applying Lemma \ref{Chill}, we infer
that $u(t)\rightarrow u_{\infty}$ strongly in $H,$ which in view of
\cite[Corollary 3.6]{CF} implies that $u(t)$ converges strongly to $u_{\infty
}$ in $V.$ Finally, Theorem \ref{th1} ensures that $u_{\infty}\in\arg\min
\Phi.$ The proof is completed.
\end{proof}

\end{document}